\documentclass[11pt]{amsart}
\usepackage{graphicx}
\newtheorem{thm}{Theorem}[section]
\newtheorem{cor}[thm]{Corollary}
\newtheorem{lem}[thm]{Lemma}

\newtheorem{exm}[thm]{Example}
\theoremstyle{definition}

\theoremstyle{remark}
\newtheorem{rem}[thm]{Remark}
\numberwithin{equation}{section}

\def\cf{\mathcal{ F}}

\def\ch{\mathcal{ H}}

\def\ct{\mathcal{ T}}

\def\cw{\mathcal{W}}

\def\bc{{\mathbb C}}

\def\bn{{\mathbb N}}

\def\br{{\mathbb R}}
\def\bt{{\mathbb T}}

\def\a{\alpha}
\def\b{\beta}
\def\g{\gamma}  
\def\d{\delta}  

\def\e{\epsilon}
\def\l{\lambda} 

\def\m{\mu}

\def\n{\nu}

\def\s{\sigma} 
\def\t{\tau}

\def\w{\omega} \def\O{\Omega}

\def\a{\alpha}

\def\bal1{\textbf{B}_1^+}

\begin{document}

\title[Weighted Laws of Large Numbers]{Weighted Laws of Large Numbers and Weighted One-Sided ergodic Hilbert Transforms on Vector Valued $L_p$-spaces}

 \author{Farrukh Mukhamedov}
\address{Farrukh Mukhamedov\\
Department of Mathematical Sciences, \& College of Science,\\
The United Arab Emirates University, Al Ain, Abu Dhabi,\\
15551, UAE} \email{{\tt far75m@yandex.ru}, {\tt
farrukh.m@uaeu.ac.ae}}

 \maketitle

\begin{abstract}
A more general notion of weight called \textit{admissible} is
introduced and then an investigation is carried out on the a.e.
convergence of weighted strong laws of large numbers and their
applications to weighted one-sided ergodic Hilbert transforms on
vector-valued $L_p$-spaces. Furthermore, the obtained results are applied to the a.e. convergence of random
modulated weighted ergodic series, which is also new in the classical setting. \\

\end{abstract}

\section{Introduction}

In the present paper, a more general notion of weight called
\textit{admissible} is introduced. For such types of weights, we
study the weighted strong laws of large numbers (SLLN) on
vector-valued $L_p$-spaces. If one considers a particular cases of
the introduced weights, we recover the earlier know results (see
for example \cite{CL2003,C}).

We point out that investigation of the convergence of weighted
strong laws of large numbers, for independent identically
distributed (iid) variables, has been started in \cite{JOP}.
Namely, for a sequence $\{w_k\}$ of positive numbers (weights)
with partial sums $W_n=\sum_{k=1}^n w_k$ (which goes to $\infty$),
they found necessary and sufficient conditions to ensure that for
every iid sequence $\{\xi_k\}$ with finite expectations the
weighted averages
$$
\frac{1}{W_n}\sum_{k=1}^n w_k\xi_k
$$
converge almost everywhere (a.e.). In \cite{LinWeb},  for a linear
contraction $T:L_p\to L_p$ the a.e. convergence or $L_p$-norm convergence of the weighted averages
$$
\frac{1}{W_n}\sum_{k=1}^n w_k\ T^k f, \ \ f\in L_p(\m)
$$
have been studied. Moreover, weighted strong laws of large numbers
(SLLN) were established for centered random variables. There are
many results devoted to the investigations of SLLN \cite{As2,As3}.

On the other hand, the investigations of  SLLN are closely related
to the  a.e. convergence of the one-sided ergodic Hilbert
transform
$$
\sum_{n=1}^\infty \frac{T^nf}{n}, \ \ f\in L_p
$$
of a contraction $T: L_p\to L_p$ (see \cite{H}).  Depending on
$p$, many papers are devoted to the convergence (i.e. a.e. or norm
convergence) of the one-sided Hilbert transform (see
\cite{AC,AsLin,C,Cot,CL09,CCL,Ga,LOT})

In \cite{BMS,CL2003,C} relations between weighted SLLN and weighted one-sided ergodic Hilbert transforms have been studied.
Here by the weighted one-sided ergodic Hilbert transform we mean the following one
$$
\sum_{n=1}^\infty \frac{a_nT^nf}{n^{\g}},
$$
where $\{a_n\}$ some bounded sequence, and $\g>0$. The case
$a_n=\l^n$, (where  $|\l|=1$) and $\g=1$ has been considered in
\cite{CCC}. Other cases have been investigated in \cite{CL05} by
means of extensions of Menchoff-Rademacher theorem. These types of
research also are closely related to the investigations of ergodic
series, i.e.
$$
\sum_{n=1}^\infty a_n f(T^nx)
$$
where $T$ is a measure-preserving transformation of $(\O,\cf,\m)$ (see \cite{A,BeW,H,K}). Recently, in \cite{Fan17}
it was proved that the last series converges a.e. if and only if
$\sum_{n=1}^\infty |a_n|^2<\infty$.

In the present paper, we obtain the a.e. convergence of weighted
one-sided ergodic Hilbert transforms, i.e.
$$
\sum_{n=1}^\infty \frac{a_nT^{n_k}f}{W_n},
$$
which is clearly more general then the known ones. It is stressed
that the obtained results extend and provide a general framework
for all known theorems in this direction. Furthermore, as an
application of the proved results, the a.e. converges of random
modulated weighted ergodic series is obtained which is even new in
the classical setting. We hope that the results of this paper open
new insight in the theory of random ergodic Hilbert transforms.

\section{Admissible weights}

Let $(\Omega,\mathcal{F},\m)$ be an arbitrary measure space with a
probability measure $\m$, and $X$ a Banach space (over real or
complex numbers). We denote the Lebesgue-Bochner space
$L_p(\O,\mathcal{F},\m;X)$ ($p\geq 1$) by $L_p(\m,X)$, and by
$L_p(\m)$ when $X$ is a scalar field, if there is no chance of
confusing the underlying  measurable space. By $\ell_p$ we denote
the standard $p$-summable sequence space. For the definitions and
properties of these spaces we refer to \cite{DS,DU}.

An increasing sequence $\{G_n\}$ is said to be a \textit{weight} if $G_1\geq 1$ and $G_n\to\infty$ as $n\to\infty$.
A weight $\{W_n\}$ is called \textit{weak $p$-admissible} ($p>1$) w.r.t. to a weight $\{G_n\}$  if one can find an increasing sequence $\{n_k\}\subset\bn$ and a sequence $\{\xi_n\}$  such that
\begin{eqnarray}\label{W1}
&&\sum_{k=1}^\infty\bigg(\frac{G_{n_k}}{W_{n_k}}\bigg)^p<\infty\\[2mm]
\label{W2}
&&\sum_{k=1}^\infty\bigg(\frac{\xi_k}{W_{n_k}}\bigg)^p<\infty
\end{eqnarray}

\begin{rem} We note that in this definition it is not necessary that the sequence $\{\xi_k\}$ equals $\{G_{n_k}\}$.
Below, we provide an example which clarifies this issue, and it also gives some way to construct weak $p$-admissible weights.
\end{rem}

\begin{exm}\label{E0} Let $p>1$ and assume that $\{G_n\}$ be a weight and $\{n_k\}\subset\bn$ be an increasing sequence such that
\begin{equation}\label{E01}
\sum_{k=1}^\infty\frac{1}{n_k}<\infty, \ \ \ \sum_{k=1}^\infty\frac{1}{(G_{n_k})^{p}}<\infty, \ \ p>1
\end{equation}
Define $W_n=n^{1/p}G_n$. Then, due to \eqref{E01}
$$
\sum_{k=1}^\infty\bigg(\frac{G_{n_k}}{W_{n_k}}\bigg)^p=\sum_{k=1}^\infty\frac{1}{n_k}<\infty
$$
Notice that
$$
\sum_{k=1}^\infty\bigg(\frac{G_{n}}{W_{n}}\bigg)^p=\infty
$$

Now choose $\xi_k=n_k^{1/p}$ which is clearly different from $\{G_{n_k}\}$. Then
$$
\sum_{k=1}^\infty\bigg(\frac{\xi_k}{W_{n_k}}\bigg)^p=\sum_{k=1}^\infty\frac{1}{G_{n_k}^p}<\infty.
$$
Hence, the weight $\{W_n\}$ is weak $p$-admissible w.r.t. $\{G_n\}$.

Now let us consider, a particular case, namely, we define  $G_n=(\ln n)^{\b+1/p}(\ln\ln n)^\g$ with $\b>0$, $\g\geq 1$. Then for $n_k=k^k$, the condition \eqref{E01} is satisfied, hence, $W_n=n^{1/p}(\ln n)^{\b+1/p}(\ln\ln n)^\g$ is weak $p$-admissible weight w.r.t. $\{G_n\}$.
\end{exm}

A weight $\{W_n\}$ is called \textit{p-admissible} w.r.t. to a weight $\{G_n\}$ if one can find an increasing sequence $\{n_k\}\subset\bn$  such that
\begin{eqnarray}\label{W3}
&&\sum_{k=1}^\infty\bigg(\frac{G_{n_k}}{W_{n_k}}\bigg)^p<\infty\\[2mm]
\label{W4}
&&\sum_{k=1}^\infty\bigg(\frac{n_{k+1}-n_k}{W_{n_k}}\bigg)^p<\infty
\end{eqnarray}

Clearly, if we define $\xi_k=n_{k+1}-n_k$, then one obviously gets
that any $p$-admissible weight is weak $p$-admissible. The set of
all $p$-admissible weights w.r.t. $\{G_n\}$ we denote by
$\cw_p(G_n)$. Examples of $p$-admissible weights are provided in
the Appendix.

\section{Weighted Strong Law of large numbers}

Let $\{W_n\}$ be a weak $p$-admissible weight w.r.t. $\{G_n\}$ with associated sequences
$\{n_k\}$ and $\{\xi_k\}$. A sequence $\{f_n\}\subset L_p(\m,X)$ is called \textit{$(W_n)$-sequence} if there is $C>0$ such that for any $m\in\bn$
\begin{eqnarray}\label{WE1}
\bigg(\int\max\limits_{n_m< n\leq n_{m+1}}\bigg\|\sum_{k=n_m+1}^n f_k\bigg\|_X^pd\m\bigg)^{1/p}\leq  C\xi_m
\end{eqnarray}

\begin{rem}\label{r31} We notice that if  $\{f_n\}\subset L_p(\m,X)$ with $A=\sup_n\|f_n\|_p<\infty$, and $\{W_n\}$ is $p$-admissible, then $\{f_n\}$ is $(W_n)$-sequence. Indeed, we get
\begin{eqnarray*}
\int\max\limits_{n_m< n\leq n_{m+1}}\bigg\|\sum_{k=n_m+1}^n f_k\bigg\|_X^p d\m&\leq &\int\max\limits_{n_m< n\leq n_{m+1}}\bigg(\sum_{k=n_m+1}^n\|f_k\|_X\bigg)^pd\m\\[2mm]
&\leq &\bigg\|\sum_{k=n_m+1}^{n_{m+1}}\|f_k\|_X\bigg\|_p^p\\[2mm]
&\leq &\bigg(\sum_{k=n_m+1}^{n_{m+1}}\|f_k\|_p\bigg)^p\\[2mm]
&\leq& A^p (n_{m+1}-n_{m})^p
 \end{eqnarray*}
 this yields the required assertion.
\end{rem}

\begin{thm}\label{T1} Let $p\geq 1$ and $\{G_n\}$ be a weight. Assume that $\{W_n\}$ is a weak $p$-admissible weight w.r.t. $\{G_n\}$, and
$\{f_n\}\subset L_p(\m,X)$ is a $(W_n)$-sequence such that
\begin{equation}\label{T11}
\sup_n\bigg\|\frac{1}{G_n}\sum_{k=1}^n f_k\bigg\|=K<\infty.
\end{equation}
Then $\frac{1}{W_n}\sum_{k=1}^n f_k$ converges a.e. Furthermore,
\begin{equation}\label{1T11}
\sup_n\frac{1}{W_n}\bigg\|\sum_{k=1}^n f_k\bigg\|_X\in L_p(\m).
\end{equation}
\end{thm}

\begin{proof} Since $\{W_n\}$ is a weak $p$-admissible weight w.r.t. $\{G_n\}$ then there is an increasing $\{n_m\}\subset\bn$  and $\{\xi_m\}$ such that \eqref{W1} and \eqref{W2} hold.

Now due to \eqref{T11} we have
\begin{eqnarray}\label{T12}
\bigg\|\frac{1}{W_{n_m}}\sum_{k=1}^{n_m} f_k\bigg\|_p&=&\frac{G_{n_m}}{W_{n_m}}\bigg\|\frac{1}{G_{n_m}}\sum_{k=1}^{n_m} f_k\bigg\|_p\\[2mm]
&\leq& K \frac{G_{n_m}}{W_{n_m}}
\end{eqnarray}

Hence, according to \eqref{W1}, the last one implies
\begin{eqnarray*}
\int\sum_{m=1}^\infty\bigg\|\frac{1}{W_{n_m}}\sum_{k=1}^{n_m} f_k\bigg\|^p_X d\m &=&
\sum_{m=1}^\infty\bigg\|\frac{1}{W_{n_m}}\sum_{k=1}^{n_m} f_k\bigg\|_p^p\\[2mm]
&\leq& \sum_{m=1}^\infty K^p \bigg(\frac{G_{n_m}}{W_{n_m}}\bigg)^p<\infty.
\end{eqnarray*}
Hence,
$$
\sum_{m=1}^\infty\bigg\|\frac{1}{W_{n_m}}\sum_{k=1}^{n_m} f_k\bigg\|^p_X<
\infty \ \ \textrm{a.e.}
$$
so,
\begin{equation}\label{T13}
\bigg\|\frac{1}{W_{n_m}}\sum_{k=1}^{n_m} f_k\bigg\|_X \to 0 \ \ \  \textrm{a.e.}
\end{equation}


Now, for $ n_m\leq n<n_{m+1}$,
\begin{eqnarray*}
\int\max_{ n_m<n\leq n_{m+1}}\bigg\|\frac{1}{W_{n}}\sum_{k=1}^{n} f_k-\frac{1}{W_{n}}\sum_{k=1}^{n_m} f_k\bigg\|_X^pd\m&\leq&
\frac{1}{W^p_{n}}\int\max_{ n_m< n\leq n_{m+1}}\bigg\|\sum_{k=n_m+1}^n f_k\bigg\|_X^pd\m\\[2mm]
&\leq& C\bigg(\frac{\xi_m}{W_{n_m}}\bigg)^p
\end{eqnarray*}
Hence, due to \eqref{W2},
\begin{eqnarray}\label{T14}
\max_{ n_m<n\leq n_{m+1}}\bigg\|\frac{1}{W_{n}}\sum_{k=1}^{n} f_k-\frac{1}{W_{n}}\sum_{k=1}^{n_m} f_k\bigg\|_X\to 0
\ \ \textrm{a.e.}
\end{eqnarray}

On the other hand, by  \eqref{T13}
$$
\bigg\|\frac{1}{W_{n}}\sum_{k=1}^{n_m} f_k\bigg\|_X\leq\frac{1}{W_{n_m}}\bigg\|\sum_{k=1}^{n_m} f_k\bigg\|_X\to 0
$$
Therefore,  \eqref{T14} implies
$$
\bigg\|\frac{1}{W_{n}}\sum_{k=1}^{n} f_k\bigg\|_X\to 0 \ \ \ \textrm{a.e.}
$$

Now, let us denote $S_n=\sum\limits_{k=1}^{n} f_k$.
Then
\begin{eqnarray*}
\int\sup_m\bigg\|\frac{S_{n_m}}{W_{n_m}}\bigg\|_X^pd\m&\leq& \int\sum_{m=1}^\infty\bigg\|\frac{S_{n_m}}{W_{n_m}}\bigg\|_X^pd\m\\[2mm]
&\leq&\sum_{m=1}^\infty K^p \bigg(\frac{G_{n_m}}{W_{n_m}}\bigg)^p<\infty,
\end{eqnarray*}
hence,
$$
\sup_m\bigg\|\frac{S_{n_m}}{W_{n_m}}\bigg\|_X\in L_p(\m)
$$

For $ n_m\leq n<n_{m+1}$,
\begin{eqnarray*}
\bigg\|\frac{S_{n}}{W_{n}}\bigg\|_X\leq \bigg\|\frac{S_{n_m}}{W_{n_m}}\bigg\|_X+\frac{1}{W_{n_m}}\bigg\|\sum_{k=n_{m}+1}^n f_k\bigg\|_X
\end{eqnarray*}

Consequently,
\begin{eqnarray*}
\bigg\|\sup_n\frac{\|S_{n}\|_X}{W_{n}}\bigg\|_p\leq \bigg\|\sup_m\frac{\|S_{n_m}\|_X}{W_{n_m}}\bigg\|_p+\bigg\|\sup_{m,n_m}\frac{1}{W_{n_m}}\bigg\|\sum_{k=n_{m}+1}^n f_k\bigg\|_X\bigg\|_p
\end{eqnarray*}

The first term of RHS of the last expression is finite, therefore, we show the finiteness of the second term. Indeed,
\begin{eqnarray*}
\bigg\|\sup_{m,n_m}\frac{1}{W_{n_m}}\bigg\|\sum_{k=n_{m}+1}^n f_k\bigg\|_X\bigg\|_p^p&\leq &
\int\sum_{m\geq 1}\frac{1}{W_{n_m}}\sup_{n_m<n\leq n_{m+1}}\bigg\|\sum_{k=n_{m}+1}^n f_k\bigg\|_X^p d\m\\[2mm]
&\leq& C\sum_{m\geq 1}\bigg(\frac{\xi_m}{W_{n_m}}\bigg)^p<\infty
\end{eqnarray*}
which completes the proof.
\end{proof}

\begin{rem} The proved theorem extends the results of \cite{CL2003,CL05} to more general weights, since in the mentioned
papers weights considered as in Examples \ref{E1}-\ref{E3} have
been studied.
\end{rem}

Due to Remark \ref{r31} from the proved theorem, we immediately
find the following corollary.

\begin{cor}\label{CT1} Let $p\geq 1$ and $\{G_n\}$ be a weight. Assume that
$\{f_n\}\subset L_p(\m,X)$ with $\sup_n\|f_n\|_p<\infty$ and
\begin{equation}\label{CT11}
\sup_{n\geq 1}\bigg\|\frac{1}{G_n}\sum_{k=1}^n f_k\bigg\|_p<\infty.
\end{equation}
Then, for any $\{W_n\}\in\cw_p(G_n)$,  the weighted means
$$\frac{1}{W_n}\sum_{k=1}^n f_k$$ converge a.e. Furthermore
$$
\sup_{n\geq 1}\frac{1}{W_n}\bigg\|\sum_{k=1}^n f_k\bigg\|_X\in L_p(\m).
$$
\end{cor}

\begin{thm}\label{T2} Let the conditions of Theorem \ref{T1} are satisfied. Assume that
\begin{equation}\label{T21}
\sum_{n=1}^\infty\frac{G_n}{W_n}\bigg(1-\frac{W_n}{W_{n+1}}\bigg)<\infty.
\end{equation}
Then the series
\begin{equation}\label{T22}
\sum_{k=1}^\infty\frac{f_k}{W_k}
\end{equation}
 converges a.e., and
$$
\sup_{n\geq 1}\bigg\|\sum_{k=1}^n \frac{f_k}{W_k}\bigg\|_X\in L_p(\m).
$$
\end{thm}

\begin{proof} Let us denote
$$
S_0=0, \ \ S_n=\sum_{k=1}^n f_k, \ \ \n\in\bn.
$$
Then by means of Abel's summation, one finds
\begin{eqnarray}\label{T23}
\sum_{k=1}^n \frac{f_k}{W_k}&=&\sum_{k=1}^n\frac{S_k-S_{k-1}}{W_k}\nonumber\\[2mm]
&=&\frac{S_n}{W_n}+\sum_{k=1}^{n-1}\bigg(\frac{1}{W_k}-\frac{1}{W_{k+1}}\bigg)S_k.
\end{eqnarray}
Due to Theorem \ref{T1},
$$
\frac{S_n}{W_n}\to 0 \ \ \ \textrm{a.e.}
$$
On the other hand, we have
\begin{eqnarray*}
\sum_{k=1}^{n}\bigg\|\bigg(\frac{1}{W_k}-\frac{1}{W_{k+1}}\bigg)S_k\bigg\|_X&
\leq& \sum_{k=1}^{n}\bigg|\frac{1}{W_k}-\frac{1}{W_{k+1}}\bigg|\|S_k\|_X\\[2mm]
&=&\sum_{k=1}^{n}\frac{G_k}{W_k}\bigg(1-\frac{W_k}{W_{k+1}}\bigg)\bigg\|\frac{S_k}{G_k}\bigg\|_X
\end{eqnarray*}

Now, by \eqref{T21}, \eqref{T11}
\begin{eqnarray*}
\int\sum_{k=1}^{\infty}\frac{G_k}{W_k}\bigg(1-\frac{W_k}{W_{k+1}}\bigg)\bigg\|\frac{S_k}{G_k}\bigg\|_X d\m\leq
\sum_{k=1}^{\infty}\frac{G_k}{W_k}\bigg(1-\frac{W_k}{W_{k+1}}\bigg)\bigg\|\frac{S_k}{G_k}\bigg\|_p.
\end{eqnarray*}
Hence, the series
$$
\sum_{k=1}^{\infty}\frac{G_k}{W_k}\bigg(1-\frac{W_k}{W_{k+1}}\bigg)\bigg\|\frac{S_k}{G_k}\bigg\|_X
$$ converges a.e. which, due to \eqref{T23}, yields the convergence of
$$
\sum_{k=1}^\infty \frac{f_k}{W_k}.
$$
From \eqref{T23}, we immediately obtain
\begin{eqnarray*}
\sup_n\bigg\|\sum_{k=1}^n \frac{f_k}{W_k}\bigg\|_X&\leq &\sup_n\frac{\|S_n\|_X}{W_n}+\sup_n\bigg\|\sum_{k=1}^{n-1}\bigg(\frac{1}{W_k}-\frac{1}{W_{k+1}}\bigg)S_k\bigg\|_X
\end{eqnarray*}
Consequently, by Theorem \ref{T1}
\begin{eqnarray*}
\bigg\|\sup_n\bigg\|\sum_{k=1}^n \frac{f_k}{W_k}\bigg\|_X\bigg\|_p&\leq &\bigg\|\sup_n\frac{\|S_n\|_X}{W_n}\bigg\|_p+\
\sum_{k=1}^{\infty}\frac{G_k}{W_k}\bigg(1-\frac{W_k}{W_{k+1}}\bigg) \bigg\|\frac{S_k}{G_k}\bigg\|_p<\infty
\end{eqnarray*}

Let us establish the norm convergence of $\sum\limits_{k=1}^\infty \frac{f_k}{W_k}$. Indeed,
\begin{eqnarray*}
\bigg\|\sum_{k=1}^n \frac{f_k}{W_k}-\sum_{k=1}^m \frac{f_k}{W_k}\bigg\|_p\leq \bigg\|\sum_{k=m+1}^{n-1}\bigg(\frac{1}{W_k}-\frac{1}{W_{k+1}}\bigg)S_k\bigg\|_p+\bigg\|\frac{S_m}{W_m}\bigg\|_p+\bigg\|\frac{S_n}{W_n}\bigg\|_p
\end{eqnarray*}
Due to Theorem \ref{T1}
$$
\bigg\|\frac{S_m}{W_m}\bigg\|_p\to 0, \ \ \bigg\|\frac{S_n}{W_n}\bigg\|_p\to 0.
$$
On the other hand, one finds
$$
 \bigg\|\sum_{k=m+1}^{n-1}\bigg(\frac{1}{W_k}-\frac{1}{W_{k+1}}\bigg)S_k\bigg\|_p\leq K
  \sum_{k=m+1}^{\infty}\frac{G_k}{W_k}\bigg(1-\frac{W_k}{W_{k+1}}\bigg)\to 0 \ \ \textrm{as} \ \ m\to\infty.
 $$
This implies that the series $\sum\limits_{k=1}^\infty \frac{f_k}{W_k}$ converges in $L_p$-norm.
\end{proof}

\begin{rem} We point out that in all proved results the
independence of the random variables $\{f_n\}$ is not required. In
\cite{BM} the a.e. convergence of the series of admissible random
variables in Orlicz spaces has been investigated.
\end{rem}

Let us provide an example for which conditions of Theorem \ref{T2} are satisfied.

\begin{exm}\label{EwA}  Let us consider $L_2(\m,\ch)$, where $\ch$ is a Hilbert space. Take any orthogonal sequence
$\{f_n\}\subset  L_2(\m,\ch)$ with $\|f_n\|_2=\sqrt{n}$. Define
$$
G_n=\bigg(\sum_{k=1}^n \|f_k\|_2^2\bigg)^{1/2}, \ \ n\in\bn.
$$
It is clear that $\{G_n\}$ is a weigh. One can see that
$$
\bigg\|\frac{1}{G_n}\sum_{k=1}^n f_k\bigg\|_2^2=\frac{1}{G_n^2}\sum_{k=1}^n \|f_k\|_2^2=1, \ \ \forall n\geq 1.
$$
Hence, the sequence $\{f_n\}$ satisfies \eqref{T11}.

Let $n_m=m^2$, $m\in\bn$ and
$$
\xi_m=\bigg(\sum_{k=n_m+1}^{n_{m+1}} \|f_k\|_2^2\bigg)^{1/2}, \ \ m\in\bn.
$$
Now define a weight $\{W_n\}$ as follows:
$$
W_n=n^{\frac{1+\e}{4}}\sqrt{n(n+1)}, \ \ 0<\e<1.
$$
Let us show that $\{W_n\}$ is a weak 2-admissible w.r.t. $\{G_n\}$. Indeed, first we notice that
$$
G_n^2=\frac{n(n+1)}{2}, \ \  \ \xi_m=G_{n_{m+1}}-G_{n_m}.
$$
Then
\begin{eqnarray*}
\sum_{m=1}^\infty\bigg(\frac{G_{n_m}}{W_{n_m}}\bigg)^2&=&\sum_{m=1}^\infty\frac{n_m(n_m+1)}{2n_m^{(1+\e)/2}n_m(n_m+1)}\\[2mm]
&=&\frac{1}{2}\sum_{m=1}^\infty\frac{1}{m^{1+\e}}<\infty
\end{eqnarray*}
From this we infer that
 \begin{eqnarray*}
\sum_{m=1}^\infty\bigg(\frac{\xi_{m}}{W_{n_m}}\bigg)^2&=&\sum_{m=1}^\infty\frac{(G_{n_{m+1}}-G_{n_m})^2}{W_{n_m}^2}\\[2mm]
&\leq& 2^2\sum_{m=1}^\infty\bigg(\bigg(\frac{G_{n_{m+1}}}{W_{n_m}}\bigg)^2+\bigg(\frac{G_{n_{m}}}{W_{n_m}}\bigg)^2\bigg)\\[2mm]
&\leq& 2^2\sum_{m=1}^\infty\bigg(\bigg(\frac{G_{n_{m+1}}}{W_{n_{m+1}}}\bigg)^2+\bigg(\frac{G_{n_{m}}}{W_{n_m}}\bigg)^2\bigg)<\infty
\end{eqnarray*}
Hence, $\{W_n\}$ is a weak 2-admissible weight.
Notice that due to  $\e\in (0,1)$
$$
\sum_{m=1}^\infty\bigg(\frac{G_{n}}{W_{n}}\bigg)^2=\frac{1}{2}\sum_{m=1}^\infty\frac{1}{m^{(1+\e)/2}}=\infty.
$$

By the following relation
$$
\max_{n_m<n\leq n_{m+1}}\bigg\|\sum_{k=n_{m}+1}^n f_k\bigg\|^2_2\leq \max_{n_m<n\leq n_{m+1}}\sum_{k=n_m+1}^n\|f_k\|_2^2\leq \xi_m^2
$$
we conclude that $\{f_n\}$ is a $(W_n)$-sequence. Therefore, according to Theorem \ref{T1} we find the a.e. convergence of
$\frac{1}{W_n}\sum_{k=1}^n f_k$.

Now, let us check the condition \eqref{T21}. Indeed,
\begin{eqnarray*}
\frac{G_{n}}{W_{n}}\bigg(1-\frac{W_n}{W_{n+1}}\bigg)&=&\frac{1}{n^{(1+\e)/4}}\bigg(1-\frac{n^{(1+\e)/4}\sqrt{n(n+1)}}{(n+1)^{(1+\e)/4}\sqrt{(n+1)(n+2)}}\bigg)\\[2mm]
&=&\frac{1}{n^{(1+\e)/4}}\frac{(n+1)^{(1+\e)/4}\sqrt{n+2}-n^{(1+\e)/4}\sqrt{n}}{(n+1)^{(1+\e)/4}\sqrt{n+2}}\\[2mm]
&\leq &\frac{1}{n^{(1+\e)/2}\sqrt{n}}\big((n+1)^{(1+\e)/4}\sqrt{n+2}-n^{(1+\e)/4}\sqrt{n}\big) \\[2mm]
&\leq &\frac{1}{n^{(1+\e)/2}}\big((n+1)^{(1+\e)/4}-n^{(1+\e)/4}\big) \\[2mm]
&\leq &\frac{1}{n^{(1+\e)/2}}\frac{1+\e}{4}n^{-(3-\e)/4} \\[2mm]
&\leq &\frac{1+\e}{4}\frac{1}{n^{\frac{5+\e}{4}}}
\end{eqnarray*}
This yields the convergence of the series
$$
\sum_{n=1}^\infty\frac{G_{n}}{W_{n}}\bigg(1-\frac{W_n}{W_{n+1}}\bigg)<\infty.
$$
Therefore,  Theorem \ref{T2}  implies the a.e. convergence of the series
$$
\sum_{k=1}^\infty\frac{f_n}{W_n}.
$$
\end{exm}

\begin{rem} We point out that for a given weight $\{G_n\}$ one can always find $p$-admissible weight $\{W_n\}$ such that
$$
\sum_{k=1}^\infty\frac{G_k}{W_k}<\infty.
$$
Indeed, if we take $W_n=n^pG_n$, the one gets the desired relation. However, this kind of weights are not interesting, since from \eqref{T11} we immediately get
$\frac{1}{W_n}\sum_{k=1}^n f_k\to 0$ in norm. Therefore, the obtained results are much more meaningful if one assumes
\begin{equation}\label{rrr}
\sum_{k=1}^\infty\frac{G_k}{W_k}=\infty
\end{equation}
while the conditions \eqref{W1}, \eqref{W2} and \eqref{T21} hold. In this setting, it is important to get \eqref{1T11} which is not obvious with \eqref{rrr}.
\end{rem}

\begin{cor}\label{CT2}
Let $1<p<\infty$ and $\{G_n\}$ be a weight. Assume that
$\{f_n\}\subset L_p(\m,X)$ with $\sup_n\|f_n\|_p<\infty$ and $\{a_n\}$ be a bounded sequence. Let
\begin{equation}\label{CT21c}
\sup_n\bigg\|\frac{1}{G_n}\sum_{k=1}^n a_kf_k\bigg\|_p<\infty.
\end{equation}
Then for any $\{W_n\}\in\cw_p(G_n)$ with \eqref{T21}  the series
\begin{equation}\label{CT22}
\sum_{k=1}^\infty\frac{a_kf_k}{W_k}
\end{equation}
 converges a.e., and
$$
\sup_{n\geq 1}\bigg\|\sum_{k=1}^n \frac{a_kf_k}{W_k}\bigg\|_X\in L_p(\m).
$$
\end{cor}

The proof immediately follows from Theorem \ref{T2} and Corollary \ref{CT1} if one takes $f_k'=a_kf_k$ for which the condition \eqref{CT21c} reduces to \eqref{CT11}.  In Appendix B,  we provided some other
kinds of examples for which the above mentioned conditions are valid.

\begin{thm}\label{T32} Let $1<p<\infty $ and let $\{G_n\}$ be a weight.
Assume that $\{f_k\}\subset L_p(\m,X)$ with $\sup_{k}\|f_k\|_p<\infty$ such that
\begin{equation}\label{T321}
\sup_{n\geq 1}\bigg\|\frac{1}{G_n}\sum_{k=1}^n f_k\bigg\|=K<\infty.
\end{equation}
Let $\{W_n\}\in\cw_p(G_{n})$ with \eqref{T21} and let $\{a_k\}$ be a bounded sequence such that
\begin{equation}\label{T322}
\sum_{k=1}^\infty\frac{|a_k-a_{k+1}|G_{k}}{W_k}<\infty.
\end{equation}
Then the series
$$
\sum_{k=1}^\infty\frac{a_kf_k}{W_k}
$$
converges a.e.
\end{thm}

\begin{proof}  As before, let us denote $S_0=0$, $S_k=\sum\limits_{j=1}^k f_j$. Then we have
\begin{equation}
\sum_{k=1}^n\frac{a_k f_k}{W_k}=\frac{a_n S_n}{W_n}+\underbrace{\sum_{k=1}^{n-1}\bigg(\frac{a_k}{W_k}-\frac{a_{k+1}}{W_{k+1}}\bigg)S_k}_J
\end{equation}

Due to the boundedness of $\{a_k\}$, by Corollary \ref{CT1}, we infer
$$
\frac{a_n S_n}{W_n}\to 0 \ \ \ \textrm{a.e.}
$$

Let us establish that the term $J$ is absolutely a.e. convergent. Indeed,
\begin{eqnarray*}
\int\bigg\|\sum_{k=1}^{n}\bigg(\frac{a_k}{W_k}-\frac{a_{k+1}}{W_{k+1}}\bigg)S_k\bigg\|_pd\m&\leq & \sum_{k=1}^{n}\bigg|\frac{a_k}{W_k}-\frac{a_{k+1}}{W_{k+1}}\bigg| \|S_k\|_p\\[3mm]
&\leq & \sum_{k=1}^{n}\frac{|a_k-a_{k+1}|}{W_{k+1}}\|S_k\|_p\\[2mm]
&&+\sum_{k=1}^{n}\frac{\|\{a_k\}\}\|_\infty|W_k-W_{k+1}|}{W_kW_{k+1}} \|S_k\|_p\\[3mm]
&\leq & K\sum_{k=1}^{n}\frac{|a_k-a_{k+1}|G_k}{W_{k}}\\[2mm]
&&+K\|\{a_k\}\}\|_\infty\sum_{k=1}^{n}\frac{G_k}{W_k}\bigg(1-\frac{W_k}{W_{k+1}}\bigg)\\[3mm]
\end{eqnarray*}
which, due to the hypothesis of Theorem, implies the desired assertion.
\end{proof}

\section{Modulated one-side ergodic Hilbert transforms on $L_p(\m,X)$}

In this section, we are going to apply the proved results of the
previous section to the a.e. convergence of the modulated
one-sided ergodic Hilbert transforms.

Let $T:L_p(\m,X)\to L_p(\m,X)$ be a power bounded operator.  Given sequences $\{a_n\}\subset\bc$, $\{n_k\}\subset\bn$ and a weight $\{W_n\}$ by \textit{the modulated one-sided ergodic  Hilbert transform} we mean the following series
$$
\sum_{k=1}^\infty\frac{a_k T^{n_k}f}{W_k}.
$$
In this section, we are going to study sufficient conditions for  the a.e. convergence of this transform.

\begin{thm}\label{T3}
Let $1<p<\infty$ and let $\{G_n\}$ be a weight and $\{n_k\}\subset\bn$ be an increasing sequence. Assume that $T:L_p(\m,X)\to L_p(\m,X)$ is a power bounded operator. If for $f\in L_p(\m,X)$  and a bounded sequence $\{a_n\}$ one has
\begin{equation}\label{CT31}
\sup_{n\geq 1}\bigg\|\frac{1}{G_n}\sum_{k=1}^n a_k T^{n_k}f\bigg\|_p<\infty.
\end{equation}
Then for any $\{W_n\}\in\cw_p(G_n)$ with \eqref{T21}  the series
\begin{equation*}
\sum_{k=1}^\infty\frac{a_k T^{n_k}f}{W_k}
\end{equation*}
 converges a.e., and
$$
\sup_{n\geq 1}\bigg\|\sum_{k=1}^n \frac{a_kT^{n_k}f}{W_k}\bigg\|_X\in L_p(\m).
$$
\end{thm}

The proof immediately follows from Corollary \ref{CT2}, if one takes $f_k=T^{n_k}f$ which is clearly bounded.

\begin{rem} We note that if one considers the weights considered in Examples \ref{E2}, \ref{E3} and $T$ is taken as an invertible
power bounded operator acting on $L_p$-spaces (classical space),
the proved Thorem recovers  the results of \cite{Cuny11}. If one
considers a measure-preserving dynamical system $(\O,\cf,\m;T)$,
then, in \cite{Fan17}, it was given necessary and sufficient
conditions for the fulfillment of \eqref{CT31} in $L_2(\m)$ for
the weight $G_n=\sum_{k=1}^n|a_k|^2$ associated with a given
sequence $\{a_n\}$.
\end{rem}

From Theorem  \ref{T32} we obtain the following interesting result.

\begin{cor}\label{CT321}
Let $1<p<\infty$ and $\{G_n\}$ be a weight and $\{n_k\}\subset\bn$ be an increasing sequence. Assume that $T:L_p(\m,X)\to L_p(\m,X)$ is a power bounded operator and for $f\in L_p(\m,X)$  one has
\begin{equation}\label{CT31}
\sup_{n\geq 1}\bigg\|\frac{1}{G_n}\sum_{k=1}^n T^{n_k}f\bigg\|_p<\infty.
\end{equation}
Assume that $\{W_n\}\in\cw_p(G_{n})$ with \eqref{T21} and
 $\{a_n\}$ is a sequence satisfying \eqref{T322}.
Then  the series
\begin{equation}\label{CT22}
\sum_{k=1}^\infty\frac{a_k T^{n_k}f}{W_k}
\end{equation}
 converges a.e.
 \end{cor}

\begin{rem} The case $G_n=n^{1-\b}$ was treated in \cite{CL05}. Other particular cases were investigated in \cite{DL,Ga81,MT94}. Our results contain more varieties of weights which allow
to produce many interesting examples.
\end{rem}

We recall that a linear contraction $T:L_1(\m,X)\to L_1(\m,X)$ is called \textit{Dunford-Schwartz} if $\|Tf\|_\infty\leq\|f\|_\infty$ for all $f\in L_1\cap L_\infty$. Any such kind of
operator induces a contraction on all $L_p$ ($1<p\leq \infty$) (see \cite{DS,K}).  From, now on, we suppose that $X$ is considered as a Hilbert space $\ch$ with inner product $\langle\cdot,\cdot\rangle$.
We are ready to formulate a result.

\begin{thm}\label{T4}
Let $\{G_n\}$ be a weight, and $\{n_k\}\subset\bn$ be an increasing sequence, and let  $\{a_n\}$ be a bounded sequence of complex numbers such that
\begin{equation}\label{1T41}
\sup_{n\geq 1}\max_{|\l |=1}\bigg|\frac{1}{G_n}\sum_{k=1}^n a_k \l^{n_k}\bigg|=K<\infty.
\end{equation}

\begin{itemize}

\item[(i)] For every contraction  $T:L_2(\m,\ch)\to L_2(\m,\ch)$  and $f\in L_2(\m,\ch)$  the series
\begin{equation}\label{T42}
\sum_{k=1}^\infty\frac{a_k T^{n_k}f}{W_k}
\end{equation}
converges a.e., for any $\{W_n\}\in\cw_2(G_n)$ with  \eqref{T21}. Moreover,
$$
\sup_{n\geq 1}\bigg\|\sum_{k=1}^n \frac{a_kT^{n_k}f}{W_k}\bigg\|_\ch\in L_2(\m).
$$

\item[(ii)]  For every Dunford-Schwartz operator $T:L_1(\m,\ch)\to L_1(\m,\ch)$  and $f\in L_p(\m,\ch)$, $1<p<2$,  the series
\begin{equation}\label{T43}
\sum_{k=1}^\infty\frac{a_k T^{n_k}f}{\tilde W_k}
\end{equation}
converges a.e., for any $\{\tilde W_n\}\in\cw_p(G_{n}^{(p)})$, where $G_{n}^{(p)}=G_n^{\frac{2(p-1)}{p}}n^{\frac{2-p}{p}}$. Moreover,
$$
\sup_{n\geq 1}\bigg\|\sum_{k=1}^n \frac{a_kT^{n_k}f}{\tilde W_k}\bigg\|_\ch\in L_p(\m).
$$
\end{itemize}
\end{thm}

\begin{proof}  (i) From \eqref{1T41} due to Theorem 2.1 \cite{BLRT} and the unitary dilation theorem,  we have
\begin{equation*}
\sup_{n\geq 1}\bigg\|\frac{1}{G_n}\sum_{k=1}^n a_k T^{n_k}\bigg\|=K<\infty.
\end{equation*}
for every contraction $T$ in Hilbert space.

Let  $T:L_2(\m,X)\to L_2(\m,\ch)$ be a contraction, and $f\in L_2(\m,\ch)$. Then
\eqref{CT21c} holds with $f_k=T^{n_k}f$. Consequently, Corollary \ref{CT2} yields a.e. convergence of the series
$$
\sum_{k=1}^\infty\frac{a_k T^{n_k}f}{W_k}
$$
and
$$
\sup_{n\geq 1}\bigg\|\sum_{k=1}^n \frac{a_kT^{n_k}f}{W_k}\bigg\|_\ch\in L_2(\m),
$$
for every $\{W_n\}\in\cw_p(G_n)$ with  \eqref{T21}.

Moreover,
$$
\bigg\|\sup_{n\geq 1}\bigg\|\sum_{k=1}^n \frac{a_kT^{n_k}f}{W_k}\bigg\|_\ch\bigg\|_2\leq CK\|f\|_2.
$$

(ii) Let us denote
\begin{equation}\label{psi1}
\psi_n(z)=\sum_{k=1}^na_k z^{n_k}.
\end{equation}
By the maximal principle and \eqref{T41},
$$
|\psi_n(z)|\leq K G_n, \ \ |z|\leq 1.
$$
Hence, for every contraction $T$ on a Hilbert space, by Theorem A \cite{RiN}
$$
\|\psi_n(T)\|\leq K G_n
$$
where $\psi_n(T)=\sum_{k=1}^na_k T^{n_k}$.

Now, assume that $T$ is a Dunfors-Schwartz operator of $L_1(\m,\ch)$, and put
$$
T_n=\sum_{k=1}^na_k T^{n_k}.
$$
Then, due to above observation $\|T_n\|_2\leq KG_n$. Moreover, we also have  $\|T_n\|_1\leq n\|\{a_k\}\|_{\ell_\infty}$.
Hence, the Riesz-Thorin interpolation Theorem \cite[VII, p. 95]{Zyg} implies that for $1<p<2$
\begin{equation}\label{T44}
\|T_n\|_p\leq \big(n\|\{a_k\}\|_{\ell_\infty}\big)^{\frac{2-p}{p}}\big(K G_n)^{\frac{2(p-1)}{p}}
\end{equation}
Denote
$$
G_{n}^{(p)}=G_n^{\frac{2(p-1)}{p}}n^{\frac{2-p}{p}}
$$
then from \eqref{T44}, for $f\in L_p(\m,\ch)$ we obtain
$$
\bigg\|\frac{1}{G_{n}^{(p)}}\sum_{k=1}^n a_k T^{n_k}f\bigg\|_p\leq K_p
$$
where $K_p=\|\{a_k\}\|_{\ell_\infty}^{\frac{2-p}{p}}K^{\frac{2(p-1)}{p}}$.
Hence, Theorem \ref{T3} implies that the series
$$
\sum_{k=1}^\infty\frac{a_k T^{n_k}f}{\tilde W_k}
$$
converges a.e. for any $\{\tilde W_n\}\in\cw_p(G_{n}^{(p)})$ with  \eqref{T21} for $G_{n}^{(p)}$, and
$$
\sup_{n\geq 1}\bigg\|\sum_{k=1}^n \frac{a_kT^{n_k}f}{\tilde W_k}\bigg\|_\ch\in L_p(\m).
$$
This completes the proof.
\end{proof}

Now, we are going to improve above theorem. To do so, we need some auxiliary facts.

\begin{lem}\label{L5}
Let $\{G_n\}$ be a weight, and $\{n_k\}\subset\bn$ be an increasing sequence, and let  $\{a_n\}$ be a bounded sequence of complex numbers such that
\begin{equation}\label{L51}
\sup_{n\geq 1}\max_{|\l |=1}\bigg|\frac{1}{G_n}\sum_{k=1}^n a_k \l^{n_k}\bigg|=K<\infty.
\end{equation}
Then for any $r\in\br$
\begin{equation}\label{T41}
\sup_{n\geq 1}\max_{|\l |=1}\bigg|\frac{1}{G_{n,r}}\sum_{k=1}^n a_k k^{ir}\l^{n_k}\bigg|\leq |r|K
\end{equation}
where
\begin{equation}\label{L52}
G_{n,r}=\frac{1}{|r|}G_n+\sum_{k=1}^{n-1}\frac{G_k}{k}.
\end{equation}
\end{lem}

\begin{proof} Denote
Using Abel's summation and \eqref{L51}, we obtain
\begin{eqnarray*}
\bigg|\sum_{k=1}^na_kk^{ir}\l^{n_k}\bigg|&=&|n^{ir}\psi_n(\l)|+\bigg|\sum_{k=1}^{n-1}\big(k^{ir}-(k+1)^{ir}\big)\psi_k(\l)\bigg|\\[2mm]
&\leq &|\psi_n(\l)|+|r|\sum_{k=1}^{n-1}\frac{|\psi_k(\l)|}{k}\\[2mm]
&\leq& KG_n+|r|K\sum_{\ell=1}^{n-1}\frac{G_k}{k}\\[2mm]
&=&|r|KG_{n,r}
\end{eqnarray*}
where $\psi_k(\l)$ is given by \eqref{psi1}.
This completes the proof.
\end{proof}

\begin{lem}\label{L6}
Let $\{G_n\}$ be a weight and $G_{n,r}$ ($r\in\br$) be given by \eqref{L52}.  Then, for any weight $\{W_n\}\in \cw_p(G_{n,r})$ and $\d>0$
$$
\bigg\{\frac{1}{\d}W_n\bigg\}\in \cw_p(G_{n,\d r}).
$$
\end{lem}

\begin{proof}  From $\{W_n\}\in \cw_p(G_{n,r})$ we infer the existence of $\{n_k\}\subset\bn$ such that
$$
\sum_{k=1}^\infty\bigg(\frac{G_{n_k,r}}{W_{n_k}}\bigg)^p<\infty, \ \ \ \  \sum_{k=1}^\infty\bigg(\frac{n_{k+1}-n_k}{W_{n_k}}\bigg)^p<\infty
$$

Take any $\d>0$, and to prove the statement, it is enough to show
\begin{equation}\label{L61}
\sum_{k=1}^\infty\bigg(\frac{G_{n_k,\d r}}{\frac{1}{\d}W_{n_k}}\bigg)^p<\infty
\end{equation}

First, we consider
\begin{eqnarray*}
\bigg(\frac{G_{n_k,\d r}}{\frac{1}{\d}W_{n_k}}\bigg)^p&=&\bigg(\frac{\frac{1}{\d |r|}G_{n_k}+\frac{1}{\d}\sum\limits_{k=1}^{n_k-1}\frac{G_k}{k}+\big(1-\frac{1}{\d}\big)\sum\limits_{k=1}^{n_k-1}\frac{G_k}{k}}{\frac{1}{\d}W_{n_k}}\bigg)^p\\[3mm]
&\leq&\bigg(\frac{G_{n_k,r}}{W_{n_k}}+\frac{|1-1/\d|}{1/\d}\frac{\sum\limits_{k=1}^{n_k-1}\frac{G_k}{k}}{W_{n_k}}\bigg)^p\\[3mm]
&\leq & \bigg(\frac{G_{n_k,r}}{W_{n_k}}+|\d-1|\frac{G_{n_k,r}}{W_{n_k}}\bigg)^p\\[2mm]
&=&(1+|\d-1|)^p\bigg(\frac{G_{n_k,r}}{W_{n_k}}\bigg)^p
\end{eqnarray*}
which yields \eqref{L61}. This completes the proof.
\end{proof}

\begin{thm}\label{T7} Let $\{G_n\}$ be a weight, and $\{n_k\}\subset\bn$ be an increasing sequence, and let  $\{a_n\}$ be a bounded sequence of complex numbers such that
\begin{equation}\label{T71}
\sup_{n\geq 1}\max_{|\l |=1}\bigg|\frac{1}{G_n}\sum_{k=1}^n a_k \l^{n_k}\bigg|=K<\infty.
\end{equation}
For any $\{W_n\}\in\cw_p(G_{n,1})$ with
\begin{equation}\label{T72}
\sum_{n=1}^\infty\frac{G_{n,1}}{W_n}\bigg|1-\frac{W_n}{W_{n+1}}\bigg|<\infty,
\end{equation}
 let $\b>0$ such that
\begin{equation}\label{T73}
\sum_{k=1}^\infty\frac{1}{k^\b W_k}<\infty.
\end{equation}
Then for any DS operator $T:L_p(\m,\ch)\to L_p(\m,\ch)$ and $f\in L_p(\m,\ch)$, $1<p<2$, the series
$$
\sum_{k=1}^\infty\frac{a_kT^{n_k}f}{k^{\b t}W_k}
$$
converges a.e., where $t=\frac{2-p}{p}$. Moreover,
\begin{equation}\label{T74}
\bigg\|\sup_{n\geq 1}\bigg\|\sum_{k=1}^n\frac{a_k}{k^{\b t}W_k}T^{n_k}f\bigg\|_\ch\bigg\|_p<\infty.
\end{equation}
\end{thm}

\begin{proof} Assume that \eqref{T71} holds, then, for DS-operator $T$, and $f\in L_2(\m,\ch)$, taking into account Lemma \ref{L5} according to the proof of (i) Theorem \ref{T4}, we have
\begin{equation}\label{pT71}
\bigg\|\sup_{n\geq 1}\bigg\|\sum_{k=1}^n \frac{a_k k^{ir}T^{n_k}f}{\tilde W_k}\bigg\|_\ch\bigg\|_2\leq |r|K\|f\|_2.
\end{equation}
for any  $\{\tilde W_n\}\in\cw_p(G_{n,r})$ with
\begin{equation*}
\sum_{n=1}^\infty\frac{G_{n,r}}{W_n}\bigg|1-\frac{W_n}{W_{n+1}}\bigg|<\infty.
\end{equation*}

Due to Lemma \ref{L6}, we may replace any $p$-admissible weight $\{\tilde W_n\}\in\cw_p(G_{n,r})$ with $p$-admissible weight $\{W_n\}\in\cw_p(G_{n,1})$ with \eqref{T72}
Hence, for every $\{W_n\}\in\cw_p(G_{n,1})$ the inequality \eqref{pT71} reduces to
\begin{equation}\label{pT72}
\bigg\|\sup_{n\geq 1}\bigg\|\sum_{k=1}^n \frac{a_k k^{ir}T^{n_k}f}{W_k}\bigg\|_\ch\bigg\|_2\leq K\|f\|_2.
\end{equation}

Now, let $\b>0$ such that \eqref{T73} holds, i.e.
$$
\sum_{k=1}^\infty\frac{1}{k^\b W_k}<\infty.
$$

For $z\in\bc$, let us denote
\begin{equation}\label{pT73}
\Phi_{n,z}(T):=\sum_{k=1}^n\frac{a_k k^{-z\b}}{W_k}T^{n_k}.
\end{equation}
Then from \eqref{pT72}
\begin{equation}\label{pT74}
\bigg\|\sup_{n\geq 1}\big\|\Phi_{n,ir}(T)f\big\|_\ch\bigg\|_2\leq K\|f\|_2, \ \ r\in\br.
\end{equation}
For $z=1+ir$, one has
$$
\sup_{n\geq 1}\big\|\Phi_{n,1+ir}(T)f\big\|_\ch\leq \sum_{k=1}^\infty\frac{|a_k|k^{-\b}\|T^{n_k}f\|_\ch}{W_k}
$$
Hence, keeping in mind \eqref{T73}, we obtain
\begin{eqnarray}\label{pT75}
\sup_{n\geq 1}\big\|\Phi_{n,1+ir}(T)f\big\|_\ch \leq \|\mathbf{\{|a_k|\}}\|_\infty\|f\|_1\sum_{k=1}^\infty\frac{1}{k^\b W_k}<\infty
\end{eqnarray}

Now, we will follow ideas of \cite{CL2003} to employ Stein's interpolation theorem.

For a bounded measurable positive integer-valued function $I$ and $z$ with $0\leq \Re(z)\leq 1$, let us define a linear operator
$$
\Phi_{I,z}f(x):=\sum_{k=1}^{I(x)}\frac{a_kk^{-z\b}}{W_k}T^{n_k}f(x)=\sum_{j=1}^{\max I}1_{I=j}(x)\sum_{k=1}^j \frac{a_kk^{-z\b}}{W_k}T^{n_k}f(x), \ \ \  f\in L_p(\m,\ch).
$$

Take any two integrable simple functions $f$ and $g$, then one check that the function
$$
F(z)=\int \langle \Phi_{I,z}f,g\rangle d\m=\sum_{j=1}^{\max I}\sum_{k=1}^j \int \frac{a_kk^{-z\b}}{W_k}\langle T^{n_k}f(x),1_{I=j}(x)g(x)\rangle d\m
$$
is continuous and bounded in the strip $\{z\in\bc: \ 0\leq \Re(z)\leq 1\}$, and analytic in its interior. Moreover, due to \eqref{pT74} and \eqref{pT75} one has
\begin{eqnarray*}
&&\|\Phi_{I,ir}f\|_2\leq \bigg\|\sup_{n\geq 1}\|\Phi_{n,ir}f\|_\ch\bigg\|_2\leq K\|f\|_2\\[2mm]
&&\|\Phi_{I,1+ir}f\|_1\leq \bigg\|\sup_{n\geq 1}\|\Phi_{n,1+ir}f\|_\ch\bigg\|_1\leq \bigg(\|\mathbf{\{|a_k|\}}\|_\infty\sum_{k=1}^\infty\frac{1}{k^\b W_k}\bigg)\|f\|_1
\end{eqnarray*}

For $1<p<2$, let $t=\frac{2}{p}-1$, so one has $\frac{1}{p}=(1-t)\frac{1}{2}+t 1$. The Stein's interpolation Theorem implies the existence of a constant $A_t$ such that for every $f\in L_p(\m,\ch)$
$$
\|\Phi_{I,t}f\|_p\leq A_t\|f\|_p
$$
which is equivalent to
$$
\bigg\|\sum_{k=1}^{I(x)}\frac{a_kk^{-t\b}}{W_k}T^{n_k}f\bigg\|_p\leq A_t\|f\|_p.
$$

For an integer $N\geq 2$, let $I_N(x)=j$ for $j$ the first integer with
$$
\bigg\|\sum_{k=1}^{j}\frac{a_kk^{-t\b}}{W_k}T^{n_k}f(x)\bigg\|_\ch=\max_{1\leq n\leq N}\bigg\|\sum_{k=1}^{n}\frac{a_kk^{-t\b}}{W_k}T^{n_k}f(x)\bigg\|_\ch.
$$

Then for $f\in L_p(\m,\ch)$, we have
\begin{eqnarray*}
\bigg\|\max_{1\leq n\leq N}\bigg\|\sum_{k=1}^{n}\frac{a_kk^{-t\b}}{W_k}T^{n_k}f(x)\bigg\|_\ch\bigg\|_p&=&\bigg\|\sum_{k=1}^{I(x)}\frac{a_kk^{-t\b}}{W_k}T^{n_k}f(x)\bigg\|_p\leq A_t\|f\|_p
\end{eqnarray*}
and letting $N\to \infty$, one concludes that
\begin{eqnarray}\label{pT76}
\bigg\|\sup_{n\geq 1}\bigg\|\sum_{k=1}^{n}\frac{a_k}{k^{t\b}W_k}T^{n_k}f\bigg\|_\ch\bigg\|_p<\infty,
\end{eqnarray}
where $t=\frac{2-p}{p}$.

Due to
$$
\bigg(\frac{G_{n_k}+\sum_{k=1}^{n_k-1}\frac{G_k}{k}}{W_{n_k}}\bigg)\in\ell_p
$$
and
$$
\frac{G_{n_k}}{W_{n_k}}\leq \frac{G_{n_k}+\sum_{k=1}^{n_k-1}\frac{G_k}{k}}{W_{n_k}}
$$
we infer that $(W_n)\in\cw_p(G_n)$. Hence, according to Theorem \ref{T4} (i) one finds the a.e. convergence of the series
$$
\sum_{k=1}^\infty\frac{a_k T^{n_k}f}{W_k}, \ \ \ f\in L_2(\m,\ch).
$$
On the other hand, from
$$
\bigg\|\frac{a_k T^{n_k}f}{k^{t\b}W_n}\bigg\|_2\leq \bigg\|\frac{a_k T^{n_k}f}{W_n}\bigg\|_2
$$
we infer that the series
\begin{eqnarray}\label{pT77}
\sum_{k=1}^\infty\frac{a_k T^{n_k}f}{k^{t\b}W_k}
\end{eqnarray}
converges a.e. for $f\in L_2(\m,\ch)$.

Therefore, the last statement with \eqref{pT76} by applying the Banach Principle, yields the a.e. convergence of the series \eqref{pT77}  for any $f\in L_p(\m,\ch)$. This completes the proof.
\end{proof}

\begin{thm}\label{T8} Let $\{G_n\}$ be a weight, and $\{n_k\}\subset\bn$ be an increasing sequence, and let  $\{a_n\}$ be a bounded sequence of complex numbers such that
\eqref{T71} is satisfied.
For any weight $\{W_n\}$ with \eqref{T72} and
\begin{equation}\label{T81}
\frac{G_{n}}{W_n}\to 0 \ \ \ n\to \infty.
\end{equation}
and every contraction $T:L_2(\m,\ch)\to L_2(\m,\ch)$  the series
$$
\sum_{k=1}^\infty\frac{a_kT^{n_k}f}{W_k}
$$
converges in operator norm. Moreover, this convergence is uniform in all contractions.
\end{thm}

\begin{proof} Given $T$, let us denote
$$
S_n(T)=\sum_{k=1}^n a_k T^{n_k}.
$$
Then from \eqref{T71} by means of the spectral theorem for unitary operators and the unitary dilation theorem, we obtain
 $$
 \|S_n(T)\|\leq K G_n, \ \ \textrm{for all} \ \ T.
 $$
On the other hand, by means of \eqref{T23}, we have
\begin{eqnarray}\label{T82}
\sum_{k=1}^n \frac{a_k T^{n_k}}{W_k}=\frac{S_n(T)}{W_n}+\sum_{k=1}^{n-1}\bigg(\frac{1}{W_k}-\frac{1}{W_{k+1}}\bigg)S_k(T).
\end{eqnarray}
Due to \eqref{T81} one has
$$
\bigg\|\frac{S_n(T)}{W_n}\bigg\|\leq K\frac{G_n}{W_n}\to 0, \ \ \ n\to\infty
$$
which converges uniformly w.r.t. $T$.

Now, let us estimate the second term in \eqref{T82}. Indeed, using \eqref{T72},
\begin{eqnarray*}
\bigg\|\sum_{k=j}^{n-1}\bigg(\frac{1}{W_k}-\frac{1}{W_{k+1}}\bigg)S_k(T)\bigg\|&\leq &\sum_{k=j}^{n-1}\bigg\|\frac{S_k(T)}{W_k}\bigg\|\bigg(1-\frac{W_k}{W_{k+1}}\bigg)\\[3mm]
&\leq &\sum_{k=j}^{n-1}\bigg\|\frac{S_k(T)}{G_k}\bigg\|\frac{G_k}{W_k}\bigg(1-\frac{W_k}{W_{k+1}}\ \bigg)\\[3mm]
&\leq & K\sum_{k=j}^{\infty}\frac{G_k}{W_k}\bigg(1-\frac{W_k}{W_{k+1}}\bigg)\to 0 \ \ \ j\to\infty.
\end{eqnarray*}
This means that the second term in \eqref{T82} is Cauchy in operator norm, uniformly in $T$. Hence, the series
$$
\sum_{k=1}^\infty\frac{a_k T^{n_k}}{W_k}
$$
converges in operator norm.
\end{proof}

\section{Random modulated one-sided ergodic Hilbert transforms}

In this section, we provide applications of the results of the previous sections to obtain random versions of that kinds of results.

\begin{thm}\label{RT1}
Let $\{G_n\}$ be a weight, and $\{n_k\}\subset\bn$ be an increasing sequence. Assume that  $\{f_n\}\subset L_p(Y, \nu)$  be a bounded sequence such that  there exists
a subset $Y^*\subset Y$ with $\nu(Y^*)=1$ when $y\in Y^*$ one has
\begin{equation}\label{RT11}
\sup_{n\geq 1}\max_{|\l |=1}\bigg|\frac{1}{G_n}\sum_{k=1}^n f_k(y) \l^{n_k}\bigg|=K<\infty.
\end{equation}
Then for every contraction  $T:L_2(\m,\ch)\to L_2(\m,\ch)$  and $g\in L_2(\m,\ch)$  the series
\begin{equation*}\label{RT12}
\sum_{k=1}^\infty\frac{f_k(y)T^{n_k}g}{W_k}
\end{equation*}
converges $\m$-a.e., for any $\{W_n\}\in\cw_2(G_n)$ with  \eqref{T21}. Moreover,
$$
\sup_{n\geq 1}\bigg\|\sum_{k=1}^n \frac{f_k(y)T^{n_k}g}{W_k}\bigg\|_\ch\in L_2(\m).
$$
\end{thm}

The proof  uses the same argument of the proof of Theorem \ref{T4}.

\begin{rem}
 We notice that in \cite{W} for a symmetric independent complex valued random variables $\{f_n\}$ the condition \eqref{RT11}
 has been established for certain weights depending on
 $\{f_n\}$. A particular case of this type of series has been
investigated in \cite{BW} for $L_2$-contractions.
\end{rem}

Let us provide some sufficient condition for the fulfillment of condition \eqref{RT11}. To do it, we will follow \cite{CL2003} and \cite{MPi}.

\begin{thm}\label{1RT}
Let $\{G_n\}$ be a weight, and let $(Y,\nu)$ be a probability space. Assume that  $\{f_n\}\subset L_2(Y,\nu)$  be an independent with $\int f_n=0$ and $\sup\limits_n\|f_n\|_2<\infty$. Suppose
that $\{n_k\}\subset\bn$ is strictly increasing sequence such that, for every $0<\a<1$ one has
\begin{equation}\label{1RT1}
\g:=\sum_{k=1}^\infty\frac{n_k^\a}{G^2_k}<\infty.
\end{equation}
Then the series
\begin{equation}\label{1RT2}
\sum_{k=1}^\infty\frac{f_k(y)\l^{n_k}}{G_k}
\end{equation}
converges $\n$-a.e uniformly in $\l$. Moreover, we have
\begin{equation}\label{1RT21}
\sup_{n\geq 1}\max_{|\l|=1}\bigg|\sum_{k=1}^n \frac{f_k(y)\l^{n_k}}{G_k}\bigg|\in L_2(\n).
\end{equation}
\end{thm}

\begin{proof} To establish this theorem, we are going to use  \cite[Corollary 1.1.2,  p.10]{MPi} with group $G$ the unit circle, the set $A=\{n_k\}$, and the independent random variables
$\xi_{n_k}=f_k$. The sequence $\{a_n\}$ is defined as follows: $a_{n_k}=1/G_k$, and $a_j=0$, if $j\notin A$.
In what follows, we will identify the unit circle with $[0,2\pi]$ with addition modulo $2\pi$.

According to the mentioned result we need to show that
$$
I(\s):=\int_0^{2\pi}\frac{\bar\s(s)ds}{s(\log \frac{8\pi}{s})^{1/2}}
$$
is finite, where $\bar\s$ is the rearrangement of $\s$, which is defined as follows (see \cite{MPi,CL2003} for details)
 \begin{eqnarray}\label{1RT3}
\s(t):=\bigg(\sum_{j\in A}|a_j|^2|1-e^{ijt}|^2\bigg)^{1/2}=2\bigg(\sum_{k=1}^\infty \frac{\sin^2\frac{n_k t}{2}}{G^2_k}\bigg)^{1/2}.
\end{eqnarray}
Now using $|\sin t|\leq 1$, $\sin^2 t\leq |\sin t|^\a\leq |t|^\a$ and \eqref{1RT1}, from \eqref{1RT3} we obtain
\begin{eqnarray}\label{1RT4}
\s(t)&\leq &2\bigg(\sum_{k=1}^\infty \frac{n_k^\a |t|^\a}{2^\a G^2_k}\bigg)^{1/2}\nonumber\\[2mm]
&=&2^{1-\a/2}|t|^{\a/2}\bigg(\sum_{k=1}^\infty \frac{n_k^\a}{G^2_k}\bigg)^{1/2}\nonumber\\[2mm]
&=&2^{1-\a/2}|t|^{\a/2}\sqrt{\g}
\end{eqnarray}
Then for the distribution of $\s$ by means of \eqref{1RT4}
\begin{eqnarray*}
m_\s(u):=\n\{t\in[0,2\pi]: \ \s(t)<u\}\geq 2^{2/\a-1}\frac{|u|^{2/\a}}{\sqrt{\g}} .
\end{eqnarray*}
Hence, the last inequality yields, for the rearrangement of $\s$
\begin{eqnarray*}
\bar\s(s):=\sup\{ t>0 : \ m_\s(t)<s\}\leq 2^{1-\a/2}|s|^{\a/2}\sqrt{\g} .
\end{eqnarray*}
Now, let us estimate
$$
I(\s)=\int_0^{2\pi}\frac{\bar\s(s)ds}{s(\log \frac{8\pi}{s})^{1/2}}\leq 2^{1-\a/2}\sqrt{\g} \int_0^{2\pi}\frac{ds}{s^{1-\a/2}(\log \frac{8\pi}{s})^{1/2}}<\infty
$$
Hence, by \cite[Corollary 1.2, p. 10]{MPi} the series
$$
\sum_{k=1}^\infty\frac{f_k(\w)\l^{n_k}}{G_k}
$$
converges a.e. uniformly in $\l$.  Moreover, the inequality (1.15) \cite{MPi} yields \eqref{1RT21}.
\end{proof}

\begin{cor}\label{C1RT}
Let the conditions of Theorem \ref{1RT} are satisfied. Then there is a subset $Y^*\subset Y$ with $\n(Y^*)=1$ such that for every $y\in Y^*$ one has
\begin{equation*}
\sup_{n\geq 1}\max_{|\l|=1}\bigg|\frac{1}{G_n}\sum_{k=1}^n f_k(y)\l^{n_k}\bigg|<\infty.
\end{equation*}
\end{cor}

\begin{proof} By Theorem \ref{1RT} there is a subset $Y^*\subset Y$ with $\n(Y^*)=1$ such that for every $y\in Y^*$ the series
$\sum_{k=1}^\infty\frac{f_k(y)\l^{n_k}}{G_k}$ converges uniformly in $\l$. This implies
$$
\sup_{n\geq 1}\max_{|\l|=1}\bigg|\sum_{k=1}^n \frac{f_k(y)\l^{n_k}}{G_k}\bigg|<\infty.
$$
Then by the Kronecker's Lemma, we immediately obtain
$$
\sup_{n\geq 1}\max_{|\l|=1}\bigg|\frac{1}{G_n}\sum_{k=1}^n f_k(y)\l^{n_k}\bigg|<\infty
$$
which is the required assertion.
\end{proof}

We notice that Corollary \ref{C1RT} gives a sufficient condition for the validity of \eqref{RT11}.\\

Now, let us get another application of Theorem \ref{T4}.

\begin{thm}\label{RT2}
Let $\{G_n\}$ be a weight and $\{n_k\}\subset\bn$ be an increasing sequence. Assume that  $\{a_n\}$ is a bounded sequence such that
\begin{equation}\label{RT21}
\sup_{n\geq 1}\max_{|\l |=1}\bigg|\frac{1}{G_n}\sum_{k=1}^n a_k\l^{n_k}\bigg|=K<\infty.
\end{equation}
Let $\a: \Omega\to\Omega$ be a measure-preserving transformation, and $\{T_\w\}$ be a family of measurable
linear contractions of $\ch$, i.e. $\|T_\w g\|_\ch\leq \|g\|_\ch$ for all $g\in \ch$. Then for every $h\in L_2(\m)$ and $g\in \ch$
the series
\begin{equation}\label{RT22}
\sum_{k=1}^\infty\frac{a_k h(\a^{n_k}(\w))T_\w T_{\a(w)}\cdots T_{\a^{n_k-1}(\w)}g}{W_k}
\end{equation}
converges $\m$-a.e. for any $\{W_n\}\in\cw_2(G_n)$ with  \eqref{T21}.
Moreover, one has
$$
\sup_{n\geq 1}\bigg\|\sum_{k=1}^n\frac{a_k h(\a^{n_k}(\w))T_\w T_{\a(w)}\cdots T_{\a^{n_k-1}(\w)}g}{W_k}\bigg\|_\ch\in L_2(\m)
$$
\end{thm}

\begin{proof}
To prove, we define a mapping $\ct: L_2(\m,\ch)\to L_2(\m,\ch)$ as follows
\begin{equation}\label{RT23}
\ct(f)(\w)=T_{\w}f(\a(w)), \ \ f=f(\w), w\in \Omega.
\end{equation}

One can see that
$$
\|\ct(f)\|_2^2=\int_\O\|T_\w f(\a(\w))\|_\ch^2d\m\leq \int_\O\|f(\a(\w))\|_\ch^2d\m= \int_\O\|f(\w)\|_\ch^2d\m=\|f\|_2^2
$$
which implies that $\ct$ is a contraction of $L_2(\m,\ch)$.

Hence, the condition \eqref{RT21} with Theorem \ref{T4} implies that
 the series
\begin{equation}\label{RT24}
\sum_{k=1}^\infty\frac{a_k\ct^{n_k}f}{W_k}
\end{equation}
converges $\m$-a.e., for any $\{W_n\}\in\cw_2(G_n)$ with  \eqref{T21}. Moreover,
$$
\sup_{n\geq 1}\bigg\|\sum_{k=1}^n \frac{a_k\ct^{n_k}f}{W_k}\bigg\|_\ch\in L_2(\m).
$$

Now, let us choose $f$ as follows:
 $$
 f(\w)=h(\w)g
 $$
 where $h\in L_2(\O,\m)$, $g\in \ch$. Then from \eqref{RT23}
 $$
 \ct(f)(\w)=f(\a(w))T_\w(g)
 $$
 which yields
 \begin{equation*}\label{RT25}
 \ct^n(f)(\w)=h(\a^n(\w))T_\w T_{\a(w)}\cdots T_{\a^{n-1}(\w)}g.
 \end{equation*}
 Consequently, from \eqref{RT24}
$$
\sum_{k=1}^\infty\frac{a_k h(\a^{n_k}(\w))T_\w T_{\a(w)}\cdots T_{\a^{n_k-1}(\w)}g}{W_k}
$$
converges $\m$-a.e.
\end{proof}

\begin{rem}
We stress that such kind of result is not known in the classical setting.
Moreover, similar kind of result can be proved  for DS operators for $L_p$-spaces as well.
 \end{rem}

If we consider a particular case, i.e. if $h$ is constant, then the last theorem yields the a.e. convergence of
\begin{equation*}\label{5t5}
\sum_{k=1}^\infty\frac{a_k T_\w T_{\a(w)}\cdots T_{\a^{n_k-1}(\w)}g}{W_k}.
\end{equation*}

Moreover, we have an other corollary of the last theorem.

\begin{cor}\label{CRT2} Assume that the conditions of Theorem \ref{RT2} are satisfied.
Let $\a: \Omega\to\Omega$ be a measure-preserving transformation, and $T:\ch\to \ch$ be a linear contractions of  $\ch$.
Then for every $h\in L_2(\m)$ and $g\in \ch$
the series
\begin{equation}\label{CRT21}
\sum_{k=1}^\infty\frac{a_k h(\a^{n_k}(\w))T^{n_k}g}{W_k}
\end{equation}
converges $\m$-a.e. for any $\{W_n\}\in\cw_2(G_n)$ with  \eqref{T21}.
Moreover,
$$
\sup_{n\geq 1}\bigg\|\sum_{k=1}^n\frac{a_k h(\a^{n_k}(\w))T^{n_k}g}{W_k}\bigg\|\in L_2(\m)
$$
\end{cor}

The proof immediately follows from Theorem \ref{RT2} if we suppose that the family $\{T_\w\}$ does not depend on $\w$, and then from \eqref{RT22} one gets the
desired convergence.

The last corollary is a combination of ergodic series and the one-sided ergodic Hilbert transforms.  This is an another kind of a.e. convergence of a random one-sided ergodic Hilbert transform of $T$ acting on $X$. Using the methods of \cite{Fan17} one can get some applications of Corollary \ref{CRT2} to hyperbolic dynamical systems. \\

 By $\bt$ we denote the unit circle in $\bc$. Then, we have following result.

 \begin{cor}\label{CRT3}  Assume that the conditions of Theorem \ref{RT2} are satisfied.
Let  $T:\ch\to \ch$ be a linear contractions of  $\ch$.
Then for every $\l \in\bt$ and $g\in \ch$
the series
\begin{equation}\label{CRT31}
\sum_{k=1}^\infty\frac{a_k \l^{n_k}T^{n_k}g}{W_k}.
\end{equation}
converges in norm of $X$, for any $\{W_n\}\in\cw_2(G_n)$ with  \eqref{T21}.
\end{cor}

\begin{proof}  Let us consider a special particular case, when $\O=\mathbb{T}$,  and $\m$ is the standard Lebesgue measure on $\bt$.
Assume that $\a:\bt\to\bt$ is given by $\a(z)=\l z$, $\l\in\bt$. It is clear that $\a$ preserves the measure $\m$. Now, we take
$h(z)=z$ in \eqref{CRT21}. Then, from Corollary \ref{CRT2} we obtain  the norm  convergence of the series \eqref{CRT31}.
This implies the norm convergence of the weighted rotated one-sided ergodic Hilbert transform associated with a contraction of $T$ (under \eqref{RT21} condition).
\end{proof}

\begin{rem} We point out that the a.e. convergence of the rotated one-sided ergodic Hilbert transforms of a contraction $T$ acting on a Hilbert space has been investigated
in \cite{CCC}, where it was considered the case $W_n=n$, $a_n=1$. Our last result gives the norm convergence for contractions on arbitrary Banach spaces. This opens a new direction
in the theory of Hlbert transforms in Banach spaces (see, \cite{CCL, Cuny}).
\end{rem}

Now combining Theorem \ref{1RT}, Corollary \ref{C1RT} and Theorem \ref{RT2} one can establish the following result.

\begin{thm}\label{RT4}
Assume $(Y,\nu)$ be a probability space, and  $\{f_n\}\subset L_2(Y,\nu)$  be an independent with $\int f_n=0$ and $\sup\limits_n\|f_n\|_2<\infty$. Let $\{G_n\}$ be a weight, and suppose
that $\{n_k\}\subset\bn$ is strictly increasing sequence such that \eqref{1RT1} holds.
 Let $\a: \Omega\to\Omega$ be a measure-preserving transformation, and $\{T_\w\}$ be a family of measurable
linear contractions of $\ch$.
Then there is a subset $Y^*\subset Y$ with $\n(Y^*)=1$ such that for every $y\in Y^*$ and  for every $h\in L_2(\m)$ , $g\in \ch$
the series
\begin{equation}\label{RT41}
\sum_{k=1}^\infty\frac{f_k(y) h(\a^{n_k}(\w))T_\w T_{\a(w)}\cdots T_{\a^{n_k-1}(\w)}g}{W_k}
\end{equation}
converges $\m$-a.e. for any $\{W_n\}\in\cw_2(G_n)$ with  \eqref{T21}.
Moreover,
$$
\sup_{n\geq 1}\bigg\|\sum_{k=1}^n\frac{f_k(y) h(\a^{n_k}(\w))T_\w T_{\a(w)}\cdots T_{\a^{n_k-1}(\w)}g}{W_k}\bigg\|_\ch\in L_2(\m)
$$
\end{thm}

\begin{rem}  In particularly, if we consider a weight $G_n=n^{\g}$, then the last theorem is a new result for this type of weight. One can recover results of \cite{BW,CL2003}.
Moreover, by choosing the sequence $\{f_n\}$ and varying weights, one can obtain a lot interesting results.
\end{rem}

  \appendix

 \section{Examples of admissible weighted}

In this section, we provide some examples of admissible weights.

\begin{exm} \label{E1} Let $G_n=n^{1-\b}$, for some $\b\in(0,1]$ and $p>1$. Then a weight given by $W_n=n^{1-\d}$ with
$\d\in[0,\frac{p-1}{p}\b\big)$ is $p$-admissible w.r.t. $\{G_n\}$.  Indeed, first we set $n_k=[k^r]+1$, where $r=1/\b$. Then
\begin{eqnarray*}
\sum_{k=1}^\infty\bigg(\frac{G_{n_k}}{W_{n_k}}\bigg)^p&=&\sum_{k=1}^\infty\frac{1}{n_k^{(\b-\d)p}}\\[2mm]
&\leq&\sum_{k=1}^\infty\frac{1}{k^{pr(\b-\d)}}<\infty
\end{eqnarray*}
since $pr(\b-\d)=(1-r\d)p>1$. Hence, the condition \eqref{W1} is satisfied. Let us check \eqref{W2}. From the definition of the sequences, one gets
\begin{eqnarray}\label{E11}
\frac{n_{k+1}-n_k}{W_{n_k}}&\leq&\frac{(k+1)^r+1-k^r}{(k^r)^{1-\d}}\nonumber\\[2mm]
&\leq&\frac{2r(k+2)^{r-1}}{k^{r(1-\d)}}\nonumber\\[2mm]
&=&2r\bigg(\frac{k+2}{k}\bigg)^{r-1}\frac{1}{k^{1-r\d}}
\end{eqnarray}
Hence,
\begin{eqnarray}\label{EW2}
\bigg(\frac{n_{k+1}-n_k}{W_{n_k}}\bigg)^p\leq(2r)^p\bigg(\frac{k+2}{k}\bigg)^{p(r-1)}\frac{1}{k^{p(1-r\d)}}
\end{eqnarray}
which, due to the choice of $\d$, implies \eqref{W2}. We point out that this kind of weight was considered in \cite{CL2003}
\end{exm}

\begin{exm}\label{E2}  Let $\tilde G_n=n^{1-\b}\ln^\g n$, for some $\b\in(0,1]$, $\g\geq 1$ and $p>1$. Then a weight given by $\tilde W_n=n^{1-\d}\ln^\g n$ with
$\d\in[0,\frac{p-1}{p}\b\big)$ is $p$-admissible w.r.t. $\{\tilde G_n\}$.  Indeed, as above, we set $n_k=[k^r]+1$, where $r=1/\b$.  Using the same argument as in Example \ref{E1} we have \eqref{W1}.
Now according to \eqref{EW2}  one finds
\begin{eqnarray*}
\bigg(\frac{n_{k+1}-n_k}{\tilde W_{n_k}}\bigg)^p&\leq &(2r)^p\bigg(\frac{k+2}{k}\bigg)^{p(r-1)}\frac{1}{k^{p(1-r\d)}\ln^{\g p} k}\\[2mm]
&\leq& (2r)^p\bigg(\frac{k+2}{k}\bigg)^{p(r-1)}\frac{1}{k^{p(1-r\d)}}
\end{eqnarray*}
and we arrive at \eqref{W2}.
\end{exm}

\begin{exm} \label{E3} Let $R_n=\frac{n^{1-\b}}{\ln^\a n}$, for some $\b\in(0,1]$, $\a\geq 1$ and $p>1$. Then a weight given by $U_n=\frac{n^{1-\d}}{\ln^\a n}$ with
$\d\in[0,\frac{p-1}{p}\b\big)$ is $p$-admissible w.r.t. $\{R_n\}$. Indeed, let us take $n_k=[k^r]+1$, where $r=1/\b$.  One can check that \eqref{W1} is
satisfied. Again using \eqref{EW2}
\begin{eqnarray*}
\frac{n_{k+1}-n_k}{ U_{n_k}}&\leq &\frac{2r(k+2)^{r-1}\ln ^\a n_k}{k^{r(1-\d)}}\\[2mm]
&\leq &\frac{2r(k+2)^{r-1}(r+1)^\a\ln ^\a k}{k^{r(1-\d)}}\\[2mm]
&=&2r(r+1)^\a\bigg(\frac{k+2}{k}\bigg)^{r-1}\frac{\ln^\a k}{k^{r(1-\d)}}.
\end{eqnarray*}
One can check that the series
$$
\sum_{k=1}^\infty\frac{\ln^{\a p} k}{k^{pr(1-\d)}}
$$
converges, if $pr(1-\d)>1$.  Hence, the condition \eqref{W2} is satisfied.
\end{exm}

\begin{exm}\label{E4}  Now we are going to provide a more general example.  Let $\{G_n\}$ be a weight such that
\begin{equation}\label{EW3}
\sum_{k=1}^\infty\frac{1}{(G_k)^{p\e}}<\infty
\end{equation}
for some $\e>0$ and $p>1$.
Then a sequence $\{W_n\}$ is given by
$$
W_n=(G_n)^{\e+1}, \ \ n\in\bn
$$
is $p$-admissible w.r.t. $\{G_n\}$. Indeed, let us define the sequence $\{n_k\}$ as follows:
$$
n_1=1, \ \ n_{k+1}=[G_{n_k}]+n_k+1, \ \ k\in\bn.
$$
Then
$$
\sum_{k=1}^\infty\bigg(\frac{G_{n_k}}{W_{n_k}}\bigg)^p=\sum_{k=1}^\infty\frac{1}{(G_{n_k})^{p\e}}<\infty
$$
which yields \eqref{W1}. Let us check the second condition. One has
\begin{eqnarray*}
\frac{n_{k+1}-n_k}{\tilde W_{n_k}}&=&\frac{[G_{n_{k}}]+1}{(G_{n_k})^{1+\e}}\\[2mm]
 &\sim& \frac{1}{(G_{n_k})^{\e}}+ \frac{1}{(G_{n_k})^{\e+1}}\\[2mm]
 &\leq & \frac{1}{(G_{n_k})^{\e}}.
 \end{eqnarray*}
 This implies the condition \eqref{W2}.

 For instance, if we takes $G_n=n^{1/p}(\ln n)^{\b+1/p}(\ln \ln n)^\g$ (with $\b,\g>0$, $p>1$), then
 $$
 \sum_{k=1}^\infty\frac{1}{(G_n)^{p}}<\infty
 $$
 hence, $W_n=G_n^2$ is $p$-admissible.
\end{exm}

\section{More examples with condition  \eqref{T21}}

Let us consider some examples for which all the above noticed conditions are satisfied.

\begin{exm}\label{E5} Let $\{G_n\}$ and $\{W_n\}$ be weights considered in Example \ref{E1}, i.e.
$G_n=n^{1-\b}$, $W_n=n^{1-\d}$, where $\b\in(0,1]$, $p>1$ and
$\d\in[0,\frac{p-1}{p}\b\big)$.
One can see that
$$
\frac{G_k}{W_k}=\frac{1}{k^{\b-\d}}
$$
due to $\b-\d<1$, we infer  $\sum_{k=1}^\infty\frac{G_k}{W_k}=\infty$, while
\eqref{T21} is satisfied. Indeed,
 \begin{eqnarray*}
\frac{G_k}{W_k}\bigg(1-\frac{W_k}{W_{k+1}}\bigg)&=&\frac{1}{k^{\b-\d}}\frac{(k+1)^{1-\d}-k^{1-\d}}{(k+1)^{1-\d}}\\[2mm]
&\leq & \frac{1}{k^{\b-\d}}\frac{(1-\d)k^{-\d}}{k^{1-\d}}\\[2mm]
&=&\frac{1-\d}{k^{1+\b-\d}}
\end{eqnarray*}
One can see that $1+\b-\d>1$, therefore,
$$
\sum_{k=1}^\infty\frac{G_k}{W_k}\bigg(1-\frac{W_k}{W_{k+1}}\bigg)<\infty.
$$
\end{exm}

\begin{exm}\label{E6} Let $\{R_n\}$ and $\{U_n\}$ be weights considered in Example \ref{E3}, i.e.
$R_n=\frac{n^{1-\b}}{\ln ^\a n}$, $U_n=\frac{n^{1-\d}}{\ln^\a n}$, where $\b\in(0,1]$, $\a\geq 1$, $p>1$ and
$\d\in[0,\frac{p-1}{p}\b\big)$.  It is clear that
$$
\sum_{k=1}^\infty\frac{R_k}{U_k}=\infty.
$$
On the other hand, we have
 \begin{eqnarray*}
\frac{R_k}{U_k}\bigg(1-\frac{U_k}{U_{k+1}}\bigg)&=&\frac{1}{k^{\b-\d}}\bigg(1-\frac{k^{1-\d}}{(k+1)^{1-\d}} \frac{\ln^\a (k+1)}{\ln^\a k}\bigg)\\[2mm]
&\leq &\frac{1}{k^{\b-\d}}\bigg(1-\frac{k^{1-\d}}{(k+1)^{1-\d}}\bigg)\\[2mm]
&\leq&\frac{1-\d}{k^{1+\b-\d}}
\end{eqnarray*}
Hence, the series
$$
\sum_{k=1}^\infty\frac{R_k}{U_k}\bigg(1-\frac{U_k}{U_{k+1}}\bigg)
$$
converges.
\end{exm}

\begin{exm}\label{E7} Let us provide an other example of $p$-admissible weight. We are going to establish that the weight considered in Example \ref{E0} is indeed $p$-admissible.

Let $G_n=(\ln n)^{\b+1/p}(\ln\ln n)^\g$ with $\b>0$, $\g\geq 1$, and $n_m=[m^r]+1$, $r=1/\a$ with $\a\in(0,1]$. Define
$$
W_n=n^{\d/p}(\ln n)^{\b+1/p}(\ln\ln n)^\g, \ \ \ \d\geq p(1-\a)+1.
$$
Then $\{W_n\}$ is $p$-admissible w.r.t. $\{G_n\}$. Indeed, we first note that
$$
\sum_{m=1}^\infty\frac{1}{n_m}<\infty
$$
therefore
$$
\sum_{k=1}^\infty\bigg(\frac{G_{n_m}}{W_{n_m}}\bigg)^p=\sum_{m=1}^\infty\frac{1}{n_m^{\d}}<\infty.
$$
Furthermore, using the same argument of Example \ref{E1}
\begin{eqnarray*}
\frac{n_{m+1}-n_m}{W_{n_m}}&=&\frac{n_{m+1}-n_m}{n_m^{\d/p}G_{n_m}}\\[2mm]
&\leq&\frac{(m+1)^r+1-m^r}{m^{(1-\t)r}G_{n_m}}\nonumber\\[2mm]
&=&2r\bigg(\frac{m+2}{m}\bigg)^{r-1}\frac{1}{m^{1-r\t}G_{n_m}},
\end{eqnarray*}
where $\t=(p-\d)/p$.
Hence,
\begin{eqnarray*}
\bigg(\frac{n_{m+1}-n_m}{W_{n_m}}\bigg)^p &\sim&  \frac{1}{m^{(1-r\t)p}G^p_{n_m}}\\[2mm]
&\sim &\frac{1}{m^{(1-r\t)p}(\ln m)^{\b p+1}(\ln\ln m)^{\g p}}.
\end{eqnarray*}
Consequently, if $(1-r\t)p\geq1$, i.e. $\d\geq p(1-\a)+1$, the series
$$
\sum_{m=1}^\infty \frac{1}{m^{(1-r\t)p}(\ln n_m)^{\b p+1}(\ln\ln n_m)^{\g p}}
$$
converges. So, $\{W_n\}$ is $p$-admissible w.r.t. $\{G_n\}$.

Let us establish that the weights $\{G_n\}$ and $\{W_n\}$  satisfy the condition \eqref{T21}.  For the sake of simplicity, we assume that $\d=1$. Then
 \begin{eqnarray}\label{E71}
\frac{G_k}{W_k}\bigg(1-\frac{W_k}{W_{k+1}}\bigg)&=&\frac{(k+1)^{1/p}G_{k+1}-k^{1/p}G_{k}}{k^{1/p}(k+1)^{1/p}G_{k+1}}
\end{eqnarray}
Now considering the derivative of
$\psi(x)=x^{1/p}(\ln x)^{\b+1/p}(\ln\ln x)^\g$, we can estimate RHS of \eqref{E71}  as follows
$$
\frac{(k+1)^{1/p}G_{k+1}-k^{1/p}G_{k}}{k^{1/p}(k+1)^{1/p}G_{k+1}}\leq\frac{C}{k^{1+1/p}}, \ \ k\in\bn
$$
for some constant. This implies the convergence of the series
$$
\sum_{k=1}^\infty \frac{G_k}{W_k}\bigg(1-\frac{W_k}{W_{k+1}}\bigg).
$$
\end{exm}

\end{document}